\documentclass{amsart}

\usepackage{amssymb}
\usepackage[utf8]{inputenc}
\usepackage{xcolor}
\usepackage{mathtools}
\usepackage{hyperref}
\usepackage{mathrsfs}
\usepackage{comment}
\usepackage{tikz}
\usepackage{tikz-cd}
\usepackage[inline]{enumitem}

\hypersetup{colorlinks=true,citecolor=cyan,linkcolor=blue,linktocpage=true}

\newcommand{\fie}{\Bbbk}
\newcommand{\FK}[1]{\mathcal{FK}_{#1}}
\newcommand{\id}{\mathrm{id}}
\newcommand{\Imod}[1]{N_{#1}}

\newcommand{\NA}{\mathscr{B}}
\newcommand{\NAn}[1]{\mathscr{B}^{#1}}
\newcommand{\ot}{\otimes }
\newcommand{\pfie}[1]{\mathbb{F}_{#1}}
\newcommand{\Rep}{\operatorname{Rep}}
\newcommand{\SG}[1]{\mathbb{S}_{#1}}
\newcommand{\sgn}{\mathrm{sgn}}
\newcommand{\yd}[1]{\prescript{#1}{#1}{\mathcal{YD}}}

\newcommand{\Z}{\mathbb{Z}}

\newtheorem{thm}{Theorem}[section]
\newtheorem{lem}[thm]{Lemma}
\newtheorem{pro}[thm]{Proposition}
\newtheorem{cor}[thm]{Corollary}

\theoremstyle{definition}

\newtheorem{exa}[thm]{Example}

\theoremstyle{remark}
\newtheorem{rem}[thm]{Remark}

\newcommand{\YD}[1]{{_{#1}^{#1}}\mathcal{YD}}
\newcommand{\semisimplify}{\boldsymbol{\mathrm{S}}} 
\newcommand{\Ver}{\mathrm{Ver}}
\newcommand{\Vect}{\mathrm{Vect}}
\newcommand{\sVect}{\mathrm{sVect}}

\newcommand{\Cent}{C}
\newcommand{\Res}{\mathrm{Res}}

\newcommand{\ssRep}{\overline{\operatorname{Rep}}\,}
\newcommand{\Hopf}{L} 
\newcommand{\supp}{\operatorname{supp}}
\newcommand{\fchar}{\mathbb{F}}
\newcommand{\N}{\mathbb{N}}
\renewcommand{\a}{a} 

\title{Higher Fomin--Kirillov algebras are infinite-dimensional}

\author[N. Andruskiewitsch et al.]{Nicolás Andruskiewitsch,
Iván Angiono,
Istvan Heckenberger,
Simon Lentner,
Christian Reiher,
Leandro Vendramin}

\address[Andruskiewitsch]{Profesor Em\'erito de la Facultad de Matem\'atica, F\'isica,
Astronom\'\i a y Computaci\'on, Universidad Nacional de C\'ordoba. 
\newline
CIEM-CONICET. 
Medina Allende s/n (5000) Ciudad Universitaria, C\'ordoba, Argentina.
\newline
Department of Mathematics and Data
Science, Vrije Universiteit Brussel, Pleinlaan 2, 1050 Brussels, Belgium.
\newline 
Shenzhen International Center for Mathematics, Southern University of Science and Technology, Shenzhen 518055, China.}
\email{nicolas.andruskiewitsch@unc.edu.ar}

\address[Angiono]{Facultad de Matem\'atica, Astronom\'ia y F\'isica,
Universidad Nacional de C\'ordoba. CIEM -- CONICET.
Medina Allende s/n (5000) Ciudad Universitaria, C\'ordoba, Argentina}
\email{ivan.angiono@unc.edu.ar}

\address[Heckenberger]{Philipps-Universit\"at Marburg,
FB Mathematik und Informatik,
Hans-Meer\-wein-Stra\ss e,
35032 Marburg, Germany}
\email{heckenberger@mathematik.uni-marburg.de}

\address[Lentner]{Okinawa Institute of Science and Technology
1919-1 Tancha, Onna-son
Okinawa 904-0411, Japan}
\email{simon@lentner.net}

\address[Reiher]{Fachbereich Mathematik, Universität Hamburg, Hamburg, Germany}
\email{Christian.Reiher@uni-hamburg.de}

\address[Vendramin]{Department of Mathematics and Data
Science, Vrije Universiteit Brussel, Pleinlaan 2, 1050 Brussels, Belgium}
\email{Leandro.Vendramin@vub.be}

\begin{document}

\begin{abstract}
We show that the Fomin--Kirillov algebras $\FK{n}$ are
infinite-dimen\-sional for $n\geq6$. This answers a question posed by
Fomin and Kirillov in~1997. To prove this, we demonstrate that
 the Nichols algebras supported on the
conjugacy class of transpositions of the symmetric group $\SG{n}$, for $n\geq6$,
are infinite-dimensional, for both relevant $2$-cocycles. For this, we develop a new  method of partially semisimplifying a Nichols algebra along an automorphism of order $p$ in characteristic $p$.
\end{abstract}

\thanks{\emph{2020 MSC.} 16S30, 17B35. 
The work of N. A. and I. A. is partially supported by 
CONICET (PIP 11220200102916CO), FONCyT-ANPCyT (PICT-2019-03660), 
and the Secyt (UNC) (Proyecto Consolidar 33620230100517CB).
The work of I.~H. and L.~V. is partially supported by Weave 
(DFG HE 5375/12-1 and FWO G043326N), S.~L. is partially supported by JSPS KAKENHI.
}

\maketitle 

\section{Introduction}
Throughout, $\fie$ denotes a field of characteristic zero. 
The Fomin--Kirillov algebras $\FK{n}$ were introduced in \cite{MR1667680} in connection with the Schubert calculus of flag manifolds. 
Fomin--Kirillov algebras also arise in non-commutative differential
geometry~\cite{MR2106930}.
Determining their dimensions has been a longstanding 
problem, see \cite[Problem~2.2]{MR1667680}. 
In this paper, we prove that $\FK{n}$ is infinite-dimensional for every $n\geq6$. This follows from a more general result on Nichols algebras of
Yetter--Drinfeld modules whose support is isomorphic, as a rack, to
the conjugacy class of transpositions in $\SG6$.

\subsection{Pointed Hopf algebras}
The dimension problem for Fomin--Kirillov algebras is closely related to the classification of finite-dimensional pointed Hopf algebras. In the lifting method of Andruskiewitsch and Schneider~\cite{MR1659895}, a central step is to determine which Yetter--Drinfeld modules over a group algebra have finite-dimensional Nichols algebras. The symmetric groups form a particularly challenging family in this classification.
For $n\geq6$, the results of \cite{MR2786171,MR2481050,MR2317945} reduce this question for simple Yetter--Drinfeld modules over $\fie\SG{n}$ to those supported on the conjugacy class of transpositions, up to an outer automorphism when $n=6$. Indeed, an outer automorphism of $\SG6$ interchanges the conjugacy classes of transpositions and products of three disjoint transpositions.

Among modules supported on the transposition class, the remaining candidates for finite-dimensional Nichols algebras are represented by the constant cocycle $-1$ and the cocycle $\chi$ defined below. For $n\geq3$, let $X=(12)^{\SG{n}}$. Given $\sigma,\tau\in X$, write $\tau=(ij)$ with $i<j$, and set
\[
\chi(\sigma,\tau)=
\begin{cases}
\phantom{-}1 & \text{if $\sigma(i)<\sigma(j)$,}\\
-1 & \text{if $\sigma(i)>\sigma(j)$.}
\end{cases}
\]
Let $V_n$ be the braided vector space with basis
$\{x_\sigma:\sigma\in X\}$ and braiding
\[
c(x_\sigma\otimes x_\tau)
=\chi(\sigma,\tau)\,x_{\sigma\tau\sigma}\otimes x_\sigma,
\qquad \sigma,\tau\in X.
\]
The Nichols algebras associated with $\chi$ and the constant cocycle $-1$
are twist-equivalent and hence have the same Hilbert
series~\cite{MR2944712}.

For $n\in\{3,4,5\}$, the Nichols algebras $\NA(V_n)$ are finite-dimensional.
In fact, $\dim\NA(V_3)=12$ and $\dim\NA(V_4)=576$ by~\cite{MR1667680,MR1800714}, while
$\dim\NA(V_5)=8294400$ was computed by Gra\~na (see for example~\cite{MR2811166}). However, it
remained open whether $\NA(V_n)$ is finite-dimensional for any $n\geq6$.

\begin{thm}
\label{thm:main}
Let $G$ be a finite group and $V\in\yd{\fie G}$. Assume that 
the support of $V$ is isomorphic as a rack
to $(12)^{\SG{6}}$. Then $\dim\NA(V)=\infty$.
\end{thm}
For every $n\geq6$, the subspace of $V_n$ spanned by
$\{x_{(ij)}:1\leq i<j\leq6\}$ is a braided subspace isomorphic to $V_6$.
Consequently, $\NA(V_6)$ embeds into $\NA(V_n)$, and
Theorem~\ref{thm:main} implies that $\dim \NA(V_n) = \infty$.
In view of \cite[Theorem~1.1]{MR2786171}, we get:

\begin{cor}
\label{cor:S6}
Let $n\geq6$ and $0\ne V\in\yd{\mathbb{C}\SG{n}}$. Then $\dim\NA(V)=\infty$. 
\end{cor}

Combining Corollary~\ref{cor:S6} with the Lifting Method~\cite{MR1659895}, 
we get:

\begin{thm}
\label{thm:pointed}
Let $n\geq6$ and let $H$ be a finite-dimensional complex 
pointed Hopf algebra with group of grouplikes $G(H)\simeq\SG{n}$. Then 
$H\simeq\mathbb{C}\SG{n}$. 
\end{thm}

\subsection{Scheme of the proof of Theorem~\ref{thm:main}}
Our proof builds on approaches developed by the authors, in \cite{HMV, AHV}, which study Nichols algebras via reduction to finite characteristic, and in \cite{2609.12895}, where quantum groups are partially semisimplified to objects in the Verlinde category over~$\mathbb{C}$.
There are, very roughly, four major steps:
\begin{enumerate}[leftmargin=*,label=\rm{(\roman*)}]
\item 
We use existing results based on filtrations and reflection theory to reduce a general rack with cocycle of the form of interest to the case of the Fomin--Kirillov algebras over $\SG 6$ and to a characteristic 3 base field;

\medbreak
\item For an automorphism $\alpha$ of order $p$ of a Hopf algebra $\Hopf$ in characteristic~$p$, one can upgrade the semisimplification functor $\Rep C_p\to \Ver_p$ in \cite{MR4486913} to a braided monoidal functor from $\Hopf$-Yetter--Drinfeld modules to a partial semisimplification of this category. For $G$ a finite group and $L=\fchar G$, this turns out to be the Deligne product of $\Ver_p$ and of the category of Yetter--Drinfeld modules over the stabilizer subgroup $G^\alpha\subseteq G$. 
\medbreak
\item In our case, we take the inner automorphism $\alpha$ acting on $\SG 6$  given by conjugation with $\a=(123)(456)$. The partial semisimplification maps a Yetter--Drinfeld module over $\SG 6$ to a Yetter--Drinfeld module over $\SG 3\times C_3\times C_2$. For this particular choice of $\alpha$, the degree~$1$ part (the conjugacy class of transpositions) is sent to zero, and the degree~$2$~and~$3$ parts become new primitive elements, whose Nichols algebra we can study; the same effect was observed in \cite{2609.12895} Section 4.5 for $v=2$.

\medbreak
\item Over the group $\SG 3\times C_3\times C_2$, all finite-dimensional Nichols algebras were classified in \cite{HV} and there is none of the form we obtained in the previous item, not even at the level of rack sizes. This concludes a proof by contradiction.
\end{enumerate}

We expect to apply the new method of partially semisimplifying a Nichols algebra along an automorphism of order $p$ in characteristic $p$ to related problems.

\subsection{Fomin--Kirillov algebras}
For an integer $n\geq3$ denote by $\FK{n}$ the algebra with generators $x_{(ij)}$, 
where $1\leq i<j\leq n$, and relations
\begin{align*}
&x_{(ij)}^{2}=0, && \text{for $1\leq i<j\leq n$,}\\
&x_{(ij)}x_{(jk)}=x_{(jk)}x_{(ik)}+x_{(ik)}x_{(ij)}, && \text{for $1\leq i<j<k\leq n$,}\\
&x_{(jk)}x_{(ij)}=x_{(ik)}x_{(jk)}+x_{(ij)}x_{(ik)}, && \text{for $1\leq i<j<k\leq n$,}\\
&x_{(ij)}x_{(kl)}=x_{(kl)}x_{(ij)}, && \text{for distinct $i,j,k,l$ with $i<j$ and $k<l$.}
\end{align*}

Several approaches have been developed to study the dimension of 
Fomin--Kirillov algebras. Fomin and Procesi~\cite{MR1774762} introduced Hopf-algebraic
tools for $\FK{n}$ and used them to obtain a decomposition
of $\FK{n}$, as a graded vector space, into a tensor product of
$\FK{n-1}$ with a certain quotient of a free algebra. In particular, the
Hilbert series of $\FK{n-1}$ divides that of $\FK{n}$. Blasiak, Liu and Mészáros~\cite{MR3552907} studied the
subalgebras of $\FK{n}$ generated by the edges of a graph on $n$
vertices. 

\medbreak
For every $n\geq3$, by definition of Nichols algebra, there is a canonical
surjection
\[
\FK{n}\to\NA(V_n),
\]
which is an isomorphism for $n\in\{3,4,5\}$. This canonical surjection together
with Theorem \ref{thm:main} yields the following result.

\begin{thm}
\label{thm:FK} The dimension of the algebra $\FK{n}$ is infinite
for all $n\geq6$.
\end{thm}

A very different proof of Theorem~\ref{thm:FK} was proposed in \cite{Barligea},
but we identified a gap in the argument, which the author has acknowledged.
 
\subsection{Use of AI}
The proof strategy used in this paper was suggested by \text{GPT-6} Astra with 4 agents. The setting there was restricted to the Fomin--Kirillov algebra over $\SG 6$.
The AI-suggested arguments were partly adopted, partly further developed, and embedded in a more general theoretical framework.
The authors take full responsibility for the mathematical content and exposition.

\subsection{Added on October 13, 2026}
Today we received a preprint by Yang Zhang (James Cook University)
presenting an alternative proof---completely different from ours--- of Theorem~\ref{thm:FK}, thought not of Theorem~\ref{thm:main}.

\section{A braided monoidal functor}
\label{section:braided_monoidal_functor}

In this section,
all groups are finite, and modules and Hopf algebras are finite-dimensional. 
We fix a field $\fchar $ of positive characteristic $p$.

Let $A$ be a finite group. We denote by $\Rep A$ the finite-dimensional representations of $A$ over $\fchar$. Then by
\cite[Section 2.2 and Remark 2.8]{MR4486913}
 there is a general process called \emph{semisimplification} which leads to a semisimple symmetric monoidal category $\ssRep A$ and a (non-exact non-faithful) symmetric monoidal functor 
$$\semisimplify\colon\,\Rep A\to \ssRep A.$$ 
In particular for the cyclic group $A=C_p$ of order $p$, $\ssRep A=\Ver_p$ is the \emph{Verlinde category}. It is well-known that $\Ver_2\cong\Vect$ and $\Ver_3\cong \sVect$, 
\cite[Example 3.2]{zbMATH07213678}. 

\medbreak
For convenience, we consider in this paper the category $\sVect$
as the full subcategory of $\yd{\fchar C_2}$ of those Yetter--Drinfeld modules $M$ over the cyclic group $C_2=\{1,\pi\}$, where for 
$M=M_1\oplus M_\pi$ with $\delta(m)=h \otimes m$ for all $m\in M_h$, the action of $\pi$ is given by $\pi m=m$ on $M_1$ and $\pi m=-m$ on $M_\pi$.

We collect some relevant facts in the following example.

\begin{exa}\label{exm:semisimplification}
The category of representations of the group $C_p=\langle a\rangle$ consists of~$p$ indecomposable modules
\begin{align*}
\Imod n &=\fchar[t]/t^n, & 1\leq n\leq p,
\end{align*}
where $\dim \Imod n = n$.
The action is given by $a v=(t+1)v$ for all $v\in \Imod n$; 
thus, $\Imod 1$ is the only simple module.
Note that as algebras $\fchar C_p\simeq \fchar[x]/x^p$ with $x=a-1$.

\medbreak
Assume now that $p=3$. For further use, we describe, following \cite{zbMATH07213678}, the
(non-faithful, non-exact) monoidal functor 
$$\semisimplify:\, \Rep C_3\to \sVect.$$
The tensor products of the indecomposable modules of $C_3$ are
$$\Imod 2\otimes \Imod 2 = \Imod 1\oplus \Imod 3,\qquad 
\Imod 2\otimes \Imod 3=\Imod 3\oplus\Imod 3,\qquad 
\Imod 3\otimes \Imod 3
=\Imod 3\oplus \Imod 3\oplus \Imod 3.
$$
Hence $\semisimplify$ sends $\Imod 3$ to zero and $\Imod 1$ and $\Imod 2$ to the one-dimensional comodules with trivial and non-trivial coaction, respectively. For each endomorphism $f$ with non-zero trace, $\semisimplify(f)$ is a non-zero morphism.
The summand $\Imod 1$ in $\Imod 2\otimes \Imod 2$ is spanned by $1\otimes t-t\otimes 1$, hence $\semisimplify$ sends the flip in $\Rep C_p$ to $-$flip in $\sVect$.
Therefore, $\semisimplify$ is a symmetric monoidal functor.
\end{exa}

It is well-known that monoidal functors send algebras to algebras; we spell this out in the setting
of our interest.
Let $\Hopf$ be a Hopf algebra (over the field $\fchar$) and $A$ a group acting on $\Hopf$ by Hopf algebra automorphisms. Then $\Hopf$ can be thought of as a Hopf algebra in the symmetric monoidal category $\Rep A$. Applying the monoidal functor $\semisimplify:\Rep A\to \ssRep A$, 
we see that the image of $\Hopf$ is a Hopf algebra $\semisimplify(\Hopf)$ in $\ssRep A$. 
The functor $\semisimplify$ upgrades to a monoidal functor 
from the category $\prescript{}{\Hopf}{\Rep A}$ of left
$\Hopf$-modules in $\Rep A$, which means $A$-equivariant $\Hopf$-representations, to modules over the image $\semisimplify(\Hopf)$ in the semisimplification of $\Rep A$
\begin{align}
\semisimplify_\Hopf:\, \prescript{}{\Hopf}{\Rep A} \longrightarrow 
\prescript{}{\semisimplify(\Hopf)}{\ssRep A},
\end{align} 
where we use the notation in \cite[3.2.7]{MR4164719}.
Similarly it upgrades to a braided monoidal functor between the respective categories
of Yetter--Drinfeld modules, see \cite{MR1456522}, \cite[Definition \,3.4.12, Remark \,3.4.14]{MR4164719}:
\begin{align}\label{formula:YDfunctor}
\semisimplify_\Hopf^\Hopf:\, \YD{\Hopf}(\Rep A) \longrightarrow
\YD{\semisimplify(\Hopf)}(\ssRep A).
\end{align}

We now apply this to the case of a finite group $G$, $\Hopf=\fchar G$, and
$A=\langle \alpha \rangle \leq \operatorname{Aut}G$, where $\alpha \in \operatorname{Aut}G$ has prime order $p = \operatorname{char} \fchar$. 
Then $\ssRep A=\Ver_p$. As for any abelian $\fchar$-linear monoidal category there is an embedding 
$\Vect\hookrightarrow \Ver_p$ by sending a vector space to the respective 
multiple of the tensor unit, called a \emph{trivial object}. 
\begin{lem}\label{lm:permutationRep}
Let $\Omega$ be a finite set with a permutation action of 
$A=\langle\alpha\rangle \simeq C_p$; then $\Omega$
decomposes as a disjoint union $\Omega_1\cup\Omega_p$ where $\Omega_i$ 
is the union of all orbits of size $i$. 
Then the associated representation $\fchar \Omega$ of $A$ is sent by $\semisimplify$ to $\fchar\Omega_1$, embedded into $\Ver_p$ as a trivial object.
\end{lem}

\begin{proof}
Evidently, $1$, $p$ are the only possible lengths of orbits, since $p$ is prime, so $\fchar\Omega$ decomposes according to the orbits into a direct sum of 1-dimensional trivial permutation representations, which are each sent to the tensor unit by $\semisimplify$, and $p$-dimensional transitive permutation representations, which are as $A$-representations isomorphic to the Jordan block $\Imod p$ and are hence sent to zero by $\semisimplify$. 
\end{proof}

We consider now $\fchar G$, which is a permutation representation on $\Omega=G$. By Lemma~\ref{lm:permutationRep} the image $\semisimplify(\fchar G)$ of the Hopf algebra $\fchar G$
in $\Ver_p$ is the Hopf algebra $\fchar H$, where $H=G^A$ is the subgroup of fixed points,  
embedded into $\Ver_p$ as trivial object. In particular, the category of $\fchar H$-modules 
in $\Ver_p$ is equivalent to the category $\Rep H \boxtimes \Ver_p$ (see~\cite[Example 3.2.4(2)]{MR5134613}), 
where objects in the latter are direct sums of tensor products of $H$-representations 
(i.\,e. objects in ${}_{\fchar H}\Vect$)
with objects
$\semisimplify(\Imod n)$ for $1\leq n\leq p-1$ in the Verlinde category.
Hence we obtain as a special case of the functor $\semisimplify_{\fchar G}^{\fchar G}$ in \eqref{formula:YDfunctor}:

\begin{thm}\label{thm:YDfunctor}
Let $G$ be a finite group, $\alpha \in \operatorname{Aut}G$ of 
order $p = \operatorname{char} \fchar$, 
$A=\langle \alpha \rangle \leq \operatorname{Aut}G$ and let $H=G^A$. 
Then there is a braided monoidal functor
$$
 \semisimplify_{G,\alpha}:\, \YD{\fchar G}(\Rep A) \longrightarrow
\YD{\fchar H}\boxtimes \Ver_p
$$
such that the following diagram of functors (with forgetful functors as vertical arrows) commutes:
\begin{equation}\label{eq:SGalpha}
\begin{tikzcd}[column sep=1cm]
\YD{\fchar G}(\Rep A)
\arrow[r,"\semisimplify_{G,\alpha}"]
\arrow[d]
& 
\YD{\fchar H}\boxtimes \Ver_p
\arrow[d]
\\
\Rep A \arrow[r,"\semisimplify"] & \Ver_p
\end{tikzcd}
\end{equation}
\end{thm}

\begin{rem}
The functor $\semisimplify_{G,\alpha}$ is $\fchar$-linear braided monoidal, and in particular it sends (graded) braided Hopf algebras to (graded)\footnote{It is well-known that additive functors between
additive categories preserve biproducts.} braided Hopf algebras. However, the image $\semisimplify_{G,\alpha}(\NA (V))$ of a (finite-dimensional) Nichols algebra 
need not be generated in degree one nor need  all of its primitives be in degree one. 
Indeed, $\semisimplify$ is not faithful, and in particular the image of $\semisimplify_{G,\alpha}$ 
of the degree one part of a Nichols algebra can be 0.
\end{rem}

We give a convenient version to use: First, if the group of automorphisms $A$ is inner, that is, given by conjugation with a subgroup $A\subseteq G$, then there is a canonical embedding $\YD{\fchar G}\to \YD{\fchar G}(\Rep A)$ by using the existing action of $A\subseteq G$ on a Yetter--Drinfeld module $V$ once more as a separate action of $A$; clearly the $G$-action is equivariant with respect to this action of $A$ on $V$ and the $G$-coaction is equivariant because of the Yetter--Drinfeld condition.

\begin{cor}
\label{cor:functor_pcase}
Let $G$ be a group and $A=\langle \a\rangle\leq G$ a subgroup of prime order ~$p$, acting on $G$ by conjugation, and $\fchar$ a field of characteristic $p$. Then there is a braided monoidal functor and a commutative diagram, with $H=\Cent_G(\a)$
\begin{equation}
\begin{tikzcd}[column sep=1cm]
\YD{\fchar G}
\arrow[r,"\semisimplify_{G,\alpha}"]
\arrow[d]
& 
\YD{\fchar H}\boxtimes \Ver_p
\arrow[d]
\\
\Rep A \arrow[r,"\semisimplify"] & \Ver_p
\end{tikzcd}
\end{equation}
\end{cor}

To further understand our construction we recall a general construction in~\cite[Corollary 4.11]{MR5091303} in the special case of $G$-Yetter--Drinfeld modules:
Over any field $\fchar$ and for any subgroup $H\subseteq G$ there is a braided lax monoidal functor 
$$ \Res:\,\YD{\fchar G}\to \YD{\fchar H}$$
such that the $H$-grading of $\Res(V)$ for $V\in\YD{\fchar G}$ is given by $\Res(V)_h=V_h$ for all $h\in H$,
the $\fchar H$-action on $\Res(V)$ is inherited from the $\fchar H$-action on $V$, and morphisms are sent to their restriction to $\Res(V)$. 
The functor is lax monoidal, because the vertical arrows in the following diagram
\begin{equation}\label{formula:lax}
\begin{tikzcd}[column sep=.2cm]
\Res\left(\bigoplus_{g\in G}V_g\right)
\otimes 
\Res\left(\bigoplus_{h\in G}W_h\right)
\arrow[r,phantom,"="]
\arrow[d,hook, dashed]
& 
\bigoplus_{g,h\in H} V_g\otimes W_h
\arrow[d,hook]
\\
\Res\left(\left(\bigoplus_{g\in G}V_g\right)\otimes\left(\bigoplus_{h\in G}W_h\right)\right)
\arrow[r,phantom,"="]
&
\bigoplus_{\substack{g,h\in G \\gh\in H}}V_g\otimes W_h
\end{tikzcd}
\end{equation}
 are typically not isomorphisms. Indeed, in general $gh\in H$ does not imply $g,h\in H$.

\begin{lem}\label{lm:functorWithRes}
Let $\a  \in G$ be of order $p$ and $A=\langle \a\rangle$. The braided monoidal functor $\semisimplify_{G,\a}$ in Corollary~\ref{cor:functor_pcase} fits into the following commutative diagram for $H=\Cent_G(a)$:

\begin{center}
\begin{tikzcd}[column sep=normal]
\YD{\fchar G}
\arrow[r]
\arrow[d,"\Res"]
&
\YD{\fchar G}(\Rep A)
\arrow[r,"\semisimplify_{G,\a}"]
\arrow[d,""] 
&
\YD{\fchar H}\boxtimes \Ver_p
\\
\YD{\fchar H}
\arrow[r]
&
\YD{\fchar H}(\Rep A)
\arrow[r,"\sim"]
&
\YD{\fchar H}\boxtimes \Rep A
\arrow[u,"\id\boxtimes \semisimplify"]
\end{tikzcd}
\end{center}
\end{lem}
\begin{proof}
Let $V\in \YD{\fchar G}$. 
On $\Res(V)=\bigoplus_{h\in \Cent_G(\a)}V_h$ the action of $A$ commutes with the Yetter--Drinfeld structure over $\fchar H$, giving rise to the interpretation of $\Res(V)$ as a Deligne product.

For any $g\in G$ with $g\not\in \Cent_G(\a)$ the summands $V_{\a^i g \a^{-i}}$ with $0\le i<p$ are permuted by the action of $A$. Hence by Lemma~\ref{lm:permutationRep} such summands are mapped by $\semisimplify$ to zero. As a consequence, the functor $\semisimplify_{G,\a}$ factors over $\Res$ to a functor (diagonal in the diagram)  
$$\YD{\fchar H}(\Rep A) 
\cong \YD{\fchar H}\boxtimes \Rep A
\rightarrow \YD{\fchar H}\boxtimes \Ver_p,$$
which obviously coincides with the functor $\semisimplify$ on the second factor.
\end{proof}

\begin{rem}\label{rem:BecomesStrict}
It is surprising that the braided lax monoidal functor $\Res$ becomes strong monoidal when composed with $\semisimplify$. Indeed, for all $g,h\in G\setminus \Cent_G(\a)$ with $gh\in \Cent_G(\a)$, the set of all pairs $(\a ^ig\a^{-i},\a ^ih\a^{-i})$, $0\le i<p$, in the direct sum \eqref{formula:lax} forms permutation representations under the diagonal conjugation action of $\a$. This representation of $\fchar C_p$ is sent to zero by~$\semisimplify$. 
\end{rem}

We spell out a concrete example in the setting of Lemma~\ref{lm:functorWithRes}.

\begin{exa}\label{exa_needed4FK}
Let $V=M(g,X)=\fchar G\otimes_{\fchar \Cent_G(g)}X$ for some $g\in G$ and a representation $X$ of $\Cent_G(g)$.
Then $\semisimplify_{G,\a}(M(g,X))$ can be described explicitly as follows.
Write the intersection $g^G\cap H$ of the conjugacy class of $g$ in $G$ with $H=\Cent_G(\a)$ as a disjoint union
\[ g^G\cap H= \bigcup_ig_i^H \]
of conjugacy classes of $H$.
For each $i$ choose a representation $X_i$ of $\Cent_G(g_i)$ which is $G$-conjugate to the representation $X$ of $\Cent_G(g)$, and restrict it to $\Cent_H(g_i)$. Then
$$\Res(M(g,X))=\bigoplus_{i} M(g_i,X_i)$$
is an $H$-representation, with $X_i$ considered as a $\Cent_H(g_i)$-module, and these actions commute with the action of $C_p=\langle\a\rangle$.

Consider now a summand $M(g_i,X_i)$. The $C_p$-representation $X_i$ decomposes (non-canonically) as $X_i=\bigoplus _{n=1}^p V_{in}\boxtimes \Imod n$, and each endomorphism $f$ of $X_i$ (including those induced by the endomorphisms of $M(g_i,X_i)$) is a sum of endomorphisms
\[ f=\sum_{n,m=1}^p f_{nm} \]
with $f_{nm}:V_{in}\boxtimes \Imod n\to V_{im}\boxtimes \Imod m$. For all $C_p$-homomorphisms 
\[ g:V_{in}\boxtimes \Imod n\to V_{im}\boxtimes \Imod m, \quad h:V_{im}\boxtimes \Imod m\to V_{in}\boxtimes \Imod n \]
with $m\ne n$ the trace of $gh$ is zero. Hence $\semisimplify (f_{nm})=0$ whenever $n\ne m$.
In the same way it follows that
$$(\id\boxtimes \semisimplify)(X_i)=\bigoplus_{n=1}^{p-1} X_i^{(n)}\boxtimes L_n
\;\in\; \Rep(\Cent_H(g_i))\boxtimes\Ver_p $$
and
\[ \semisimplify(f)=\sum_{n=1}^{p-1}\semisimplify(f_{nn})\in \bigoplus_{n=1}^{p-1}\mathrm{End} (X_i^{(n)})\boxtimes \id_{L_n}.\]
The $\Cent_H(g_i)$-representations  $X_i^{(n)}$ can also be realized as subquotients of the $a$-invariants or the $a$-coinvariants of $X_i$, see \cite[Example 4.7]{MR4237968}.
Alternatively, consider $f_{nn}$ as an endomorphism of $V_{in}$ with coefficients in $\fchar[a]$, by  and specialize it to $a=1$. 
Altogether 
\begin{align*}
\semisimplify_{G,\a}:\quad
\YD{\fchar G}&\longrightarrow \YD{\fchar H}\boxtimes \Ver_p,\\
M(g,X)
&\longmapsto 
\bigoplus_i\bigoplus_{n=1}^{p-1}
M(g_i,X_i^{(n)})\boxtimes \semisimplify(\Imod n).
\end{align*}
In particular for $p=2$ we get representations $X_i^{(1)}$ of $\Cent_H(g_i)$ in vector spaces and for $p=3$ we get representations of $\Cent_H(g_i)$ in super vector spaces with even part $X_i^{(1)}$ and odd part $X_i^{(2)}$, respectively. As explained at the very beginning of this section, for $p=3$ we can compose our functor further with the braided monoidal embedding $\sVect\hookrightarrow \YD{\fchar C_2}$. Then we can recombine the Deligne product and the composite functor reads

\begin{equation} 
\label{formula:pthree}
\begin{aligned}
\semisimplify_{G,\a}:\quad
\YD{\fchar G}&\longrightarrow \YD{\fchar(H\times C_2)},\\
M(g,X)
&\longmapsto 
\bigoplus_i
M(g_i,X_i^{(1)}) \,\oplus\,
M(g_i\pi,X_i^{(2)}\varpi).
\end{aligned}
\end{equation}
where $\pi$ generates $C_2$ and $\varpi$ is the sign representation of $C_2$.
\end{exa}

\section{The \texorpdfstring{$\FK{3}$}{FK3}-case: a toy example}

Consider $G=\SG 3$ and the Yetter--Drinfeld module $V_\fie=M((12),\sgn)$ with  $\sgn$ the nontrivial $1$-dimensional representation of the centralizer $\Cent_G((12))=\langle (12)\rangle$ given by $\sgn((12))=-1$. 
The Nichols algebra $\NA(V_\fie)$ is described in \cite[Examples 6.2 and 6.4(1)]{MR1800714} and coincides with the Fomin--Kirillov algebra $\FK 3$. 

Set $a=(123)$ with $H=\Cent_G(a)=\langle (123)\rangle$ and $\pfie 3$ the prime field of characteristic~$3$. We fix the $\Z$-order $V_\Z$ given by the $\Z$-span of  the basis $x_{12},x_{23},x_{13}$, and reduce modulo $3$ to $V:=\pfie 3\otimes_\Z V_\Z$, so in view of Lemma \ref{lem:inequality}
\[ \pfie3 \ot_\Z \NA (V_\Z) \twoheadrightarrow
\NA(V).
\]
The semisimplification functor in Equation~\eqref{formula:pthree}
yields an $\N_0$-graded Hopf algebra 
\[\overline{\NA}=
\semisimplify_{\SG 3,a}\big({\pfie3 \ot_\Z \NA (V_\Z)}\big)
\;\;\in\;\;
\YD{\pfie3 (H\times C_2)}.
\]
Note that $\semisimplify_{\SG 3,a}(V)=0$ since $H\cap (12)^{\SG 3}=\emptyset$.
Hence $\semisimplify_{\SG 3,a}(\NAn 2(V))$ consists of primitive elements. Explicitly,
$U_\fie=\NA^2(V_\fie)$ has $2$-dimensional homogeneous components of $\SG 3$-degree $(123)$ and $(132)$ respectively, spanned by the monomials
\begin{align*}    
&x_{12}x_{23},&& x_{23}x_{13},&& x_{13}x_{12}=x_{12}x_{23}-x_{23}x_{13},\\
&x_{23}x_{12},&& x_{12}x_{13},&& x_{13}x_{23}=x_{23}x_{12}-x_{12}x_{13}.
\end{align*}
see e.g. \cite[(6.6)]{MR1800714}.
We compute the action of $a=(123)$ on
these two homogeneous components in the given $\Z$-basis for $U_\Z=\NA^2(V_\Z)$
\[
\begin{pmatrix}
0 & 1 \\
-1 & -1
\end{pmatrix},\qquad 
\begin{pmatrix}
-1 & -1 \\
1 & 0
\end{pmatrix}.
\]
In both cases, this is the unique 
simple
$2$-dimensional $C_3$-representation $Z$ over the ring $\Z$, so $U_\Z=M((123),Z)\oplus M((132),Z)$. The reduction modulo $3$ of $Z$ is the indecomposable $\Imod 2$, and its semisimplification is $L_2$ in $\Ver_3$, which we identified with the $1$-dimensional odd vector space in $\sVect$ or the $C_2$-Yetter--Drinfeld module $M(\pi,\varpi)$. Altogether by Example~\ref{exa_needed4FK},
semisimplification yields a Yetter--Drinfeld module
\[
\overline{U}=\semisimplify_{\SG 3,a} \big(\pfie 3 \ot_\Z \NA^2 (V_\Z)\big)
=M\big((123)\pi,\varpi\big) \oplus 
M\big((132)\pi,\varpi\big),
\]
with trivial action of $H$. Its Nichols algebra is an
$\N_0$-graded exterior algebra $\NA (\overline{U})=\Lambda \overline{U}$
with $\deg \overline U=2$.
Since $\NAn n(\FK 3)=0$ for $n\ge 5$, it follows that $\semisimplify_{\SG 3,a}(\FK 3)=\NA (\overline U)$.

\section{The proof of Theorem~\ref{thm:main}}

We prepare the proof of one of our key results with some lemmas.

\begin{lem}
\label{lem:rack_homomorphism}
Let $X$ and $Y$ be conjugacy classes of groups $G$ and $H$, respectively. If $G=\langle X\rangle$, $H=\langle Y\rangle $, and 
$Z(H)$ is trivial, then any surjective
rack homomorphism $X\to Y$ 
extends uniquely to a surjective group homomorphism 
$G\to H$.
\end{lem}

\begin{proof}
Uniqueness follows from $G=\langle X\rangle$.
Since $G=\langle X\rangle$ and $H=\langle Y\rangle$, conjugation induces surjective group homomorphisms $\pi_X:G\to \operatorname{Inn}(X)$ and
$\pi_Y:H\to \operatorname{Inn}(Y)$. Since $Z(H)$ is trivial, $\pi_Y$ is an isomorphism. 
Thus the claim follows by composing $\pi_X$ with
the surjective homomorphism 
$\operatorname{Inn}(X)\to H$ in~\cite[Corollary 1.10]{MR1994219}.
\end{proof}

In the next lemma, for each subgroup $H$ of a group $G$ and for each $x\in G$ we write $x^H$ for the orbit of the conjugation action of $H$ on $x$.

\begin{lem}\label{lem:kercapGG}
Let $f\colon G\to Q$ be a surjective group
homomorphism such that $\ker f$ intersects $[G,G]$ trivially.
Then for every $a\in G$,
\[
C_G(a)=f^{-1}(C_Q(f(a))).
\]
Moreover, for each $x\in G$, $f$ induces a bijection between $x^{C_G(a)}$ and $f(x)^{C_Q(f(a))}$.
\end{lem}

\begin{proof}
If $f(g)$ commutes with $f(a)$, then
$gag^{-1}a^{-1}\in\ker f\cap[G,G]=\{1\}$. This implies the first claim.
Let now $x\in G$. Then $f\colon x^G\to f(x)^Q$ is surjective and injective:
if $f(u)=f(v)$, then $uv^{-1}\in\ker f\cap[G,G]=\{1\}$, so $u=v$.
Since $f(C_G(a))=C_Q(f(a))$, the claim follows.
\end{proof}

\begin{lem}
\label{lem:G_properties}
  Let $\varphi: G\to \SG 6$ be a surjective group homomorphism, and let $X$ be a conjugacy class of $G$. Assume that 
  $G=\langle X\rangle $ 
  and that 
  $\varphi|_X:X\to (12)^{\SG 6}$ is bijective.
  \begin{enumerate}
      \item The map $\varphi|_X:X\to (12)^{\SG 6}$ is an isomorphism of racks.
      \item $\ker \varphi \cap [G,G]=\{1\}$.
      In particular,
      \[ xy=yx \text{ in $G$} \Longleftrightarrow  \varphi(x)\varphi(y)=\varphi(y)\varphi(x) \text{ in $\SG 6$.}
      \]
      \item
      For each $\a,x\in G$,
      $\varphi$ restricts to a bijection $x^{C_G(\a)}\to \varphi(x)^{C_{\SG 6}(\varphi(a))}$.
     \item Let $\a\in [G,G]$ with $\varphi(\a)=(123)(456)$. Then
      \[C_G(\a)=\varphi^{-1}\big(\langle (123),(456),(14)(25)(36)\rangle\big). \]
  \end{enumerate}
\end{lem}

\begin{proof}
    (1) is straightforward. (2) By \cite[Proposition 3.2]{MR2799090} there is a short
    exact sequence 
    \[
    1\to\Z\to G_X\to\SG6\to 1,
    \]
    where $G_X$ is the enveloping group of $X$. Since $G=\langle X\rangle$, it follows that 
    $G$ is a cyclic extension of $\SG6$. This implies the claim. 
    (3) follows from (2) and Lemma~\ref{lem:kercapGG}. 
    (4) is a consequence of (3) and of elementary calculations in~$\SG 6$. 
\end{proof}

We will also need a claim on braided Hopf algebras. Recall from \cite[Definition 2.1]{HV} that a tuple $(U_1,\dots,U_n)$ of finite-dimensional absolutely simple objects in $\yd{\fchar H}$ for a field $\fchar$ and a group $H$ is \emph{braid-indecomposable}
if there is no decomposition $\{1,\dots,n\}=I\cup J$ into non-empty disjoint subsets such that $(\id -c^2)(U_i\otimes U_j)=0$ for all $i\in I$, $j\in J$.
Note that if $(\id-c^2)(U_i\otimes U_j)=0$ then $gh=hg$ for all $g\in \supp U_i$, $h\in \supp U_j$.

A subspace $U$ of a braided Hopf algebra is called \emph{primitive} if $\Delta (u)=1\otimes u+u\otimes 1$ for all $u\in U$.

\begin{lem} \label{lem:largeU}
  Let $H$ be a finite group and $B\in \yd{\fchar H}$ be a braided Hopf algebra. 
  Assume that $B$ contains a subspace
  \[ U=U_1\oplus \cdots \oplus U_n \]
  of primitives, where $n\ge 3$, the $U_i$ are absolutely simple objects in $\yd{\fchar H}$, 
  $\dim U_i\ge 2$ for all $i$, and $\dim U_i>2$ for at least one $i$. 
  If $(U_1,\dots,U_n)$ is braid-indecomposable and $\supp U$ generates $H$, 
  then $\dim B=\infty$.
\end{lem}

\begin{proof}
  Since $U$ is a sum of primitives, the subalgebra of $B$ generated by $U$ is a braided Hopf algebra.
  It has a coradical filtration, and the corresponding coradically graded Hopf algebra contains the Nichols algebra of $U$. 
  The dimension restrictions and the braid-indecomposability of $U$ imply that $\dim \NA (U)=\infty$ by the Classification Theorem of \cite{HV}.
  Therefore $\dim B=\infty$.
\end{proof}

One of the main ingredients of our proof of Theorem~\ref{thm:main} is the following proposition.

\begin{pro}
\label{pro:char3}
Let $\fchar$ be a field of characteristic three, $G$ a finite group and $V\in\yd{\fchar G}$. Assume that 
there is an isomorphism $\varphi: \supp V\to (12)^{\SG{6}}$ of racks, and the following conditions hold:
\begin{gather}
\dim V_g=1
\shortintertext{for all $g\in\supp V$, and}
\label{eq:char3_1}\dim\NAn2(V)_{\varphi^{-1}( (12) )\varphi^{-1}( (23))}=2,\\
\label{eq:char3_2}\dim\NAn2(V)_{\varphi^{-1}((12))\varphi^{-1}((34))}=1.
\end{gather}
Then $\dim\NA(V)=\infty$.
\end{pro}

\begin{proof}
Let $\widetilde{G}$ be the subgroup of $G$ generated by
$\supp V$. Then the restriction of the action of $G$ on $V$ to
$\widetilde{G}$ yields
a Yetter--Drinfeld module $\Res V\in \yd{\fchar \widetilde{G}}$, and
$\dim \NA(\Res V)=\dim\NA(V)$. Therefore we may assume
that $G$ is generated by $\supp V$.

Let $\varphi\colon\supp V\to (12)^{\SG 6}$ be a rack isomorphism. By
Lemma~\ref{lem:rack_homomorphism},
there is a surjective group homomorphism $\overline\varphi\colon G\to\SG{6}$ with $\overline \varphi|_{\supp V}=\varphi$.

For $1\le i<j\le6$, let $g_{ij}=g_{ji}=\varphi^{-1}((ij))$, and let
$\a=g_{12}g_{23}^{-1}g_{45}g_{56}^{-1}\in G$. Then $\a\in [G,G]$, $\overline \varphi(\a)=(123)(456)$, and 
\[
\a^3\in \ker \overline \varphi \cap [G,G]=\{1\}
\]
by Lemma~\ref{lem:G_properties}(2). 
Hence $\a$ has order three. 


Assume now that $\dim \NA (V)<\infty$. We are going to apply the functor $\semisimplify_{G,\a}$ of
Equation~\eqref{formula:pthree}
to $\NA (V)$, and to conclude the proposition using Lemma~\ref{lem:largeU}.



For each $1\leq i<j\leq 6$, let $x_{ij}\in V$
be a basis vector of $V_{g_{ij}}$. Moreover, for all $n\ge 1$ let
\[ f_n:(\supp V)^n\to \SG 6,\quad (h_1,\dots,h_n)\mapsto
\varphi(h_1)\cdots \varphi(h_n). \]
First note that $\supp V\cap C_G(\a)=\emptyset$, and hence
$\semisimplify_{G,\a}(V)=0$ by Equation~\eqref{formula:pthree} of Example~\ref{exa_needed4FK}. 

Next observe that
\begin{align*}
f_2^{-1}( (123))&=\{ (g_{12},g_{23}),(g_{23},g_{13}),(g_{13},g_{12}) \},\\
f_2^{-1}( (12)(34))&=\{ (g_{12},g_{34}),(g_{34},g_{12}) \}.
\end{align*}
Hence
\[
V^{\otimes2}_{g_{12}g_{23}}=\operatorname{span}\{x_{12}x_{23},x_{23}x_{13},x_{13}x_{12}\},
\quad
V^{\otimes2}_{g_{12}g_{34}}=\operatorname{span}\{x_{12}x_{34},x_{34}x_{12}\}.
\]
In particular, $V^{\otimes2}_{g_{12}g_{23}}$ is a free $\fchar \langle
\a \rangle $-module.
Since
$\dim\,\NAn2(V)_{g_{12}g_{23}}=2$ by the assumption in
\eqref{eq:char3_1}, we conclude that
$\NAn2(V)_{g_{12}g_{23}}$ is an indecomposable $\fchar \langle
\a\rangle $-module.
In a similar way we note that $f_3^{-1}((14)(25)(36))$ consists of the
six triples
containing each of the elements $g_{14}$, $g_{25}$, $g_{36}$ as
entries. Note that
these elements are
pairwise commuting.
Since
$\dim\,\NAn2(V)_{g_{ij}g_{kl}}=1$ for all pairwise distinct
$i,j,k,l$ in $\{1,2,\dots,6\}$
as a consequence of the assumption in \eqref{eq:char3_2},
we conclude that the $\fchar \langle \a \rangle $-module $\NAn3(V)_{g_{14}g_{25}g_{36}}$ 
is 1-dimensional.

Let
\begin{align*}
U_2=G\cdot \NAn2(V)_{g_{12}g_{23}}\in\yd{\fchar G}.
\end{align*}
Note that
\[ (g_{12}g_{23})^G\cap C_G(\a)=\{ g_{12}g_{23},g_{23}g_{12},g_{45}g_{56},
g_{56}g_{45}\}
\]
by 
Lemma~\ref{lem:G_properties}(2). By (3) and (4) of the same lemma, 
\[
(g_{12}g_{23})^{C_G(\a)}=\{g_{12}g_{23}, g_{45}g_{56}\},
\quad 
(g_{23}g_{12})^{C_G(\a)}=\{g_{23}g_{12}, g_{56}g_{45}\}.
\]
Since $\NAn2(V)_{g_{12}g_{23}}$ is indecomposable 2-dimensional 
as an $\fchar\langle\a\rangle$-module, 
Equation~\eqref{formula:pthree} implies that 
\[ 
\semisimplify_{G,\a}(U_2)=\overline{U}_2'\oplus \overline{U}_2''\in\yd{\fchar(C_G(\a)\times C_2)}.
\]
Moreover, both $\overline{U}'_2$ and $\overline{U}''_2$ are
2-dimensional absolutely irreducible objects, and 
\[
\supp\overline{U}_2'=\{g_{12}g_{23}\pi, g_{45}g_{56}\pi\},\quad 
\supp\overline{U}_2''=\{g_{23}g_{12}\pi, g_{56}g_{45}\pi\}.
\]

Similarly, let 
\[
U_3=G\cdot \NAn3(V)_{g_{14}g_{25}g_{36}}\in\yd{\fchar G}. 
\]
As before, 
\[ 
(g_{14}g_{25}g_{36})^G\cap C_G(\a)=\{ g_{14}g_{25}g_{36},g_{15}g_{26}g_{34},g_{16}g_{24}g_{35}\}=(g_{14}g_{25}g_{36})^{C_G(\a)}.
\]
Since $\NAn3(V)_{g_{14}g_{25}g_{36}}$ is 1-dimensional 
as an $\fchar\langle\a\rangle$-module, 
Equation~\eqref{formula:pthree} implies that 
\[ 
\semisimplify_{G,\a}(U_3)=\overline{U}_3\in\yd{\fchar(C_G(\a)\times C_2)}
\]
is a 3-dimensional 
absolutely irreducible object with $\supp\overline{U}_3=(g_{14}g_{25}g_{36})^{C_G(\a)}$.

Note that $\supp\overline{U}_3$ commutes neither with $\supp \overline{U}'_2$ nor with $\supp \overline{U}''_2$ element-wise. Hence
the triple $(\overline{U}_2',\overline{U}_2'',\overline{U}_3)$ is braid-indecomposable. By 
Lemma~\ref{lem:largeU} applied to the group generated by $\supp\overline{U}'_2\cup\supp\overline{U}_2''\cup\supp\overline{U}_3$, 
we conclude that $\dim\semisimplify_{G,\a}(\NA(V))=\infty$. Therefore 
$\dim\NA(V)=\infty$.
\end{proof} 

Recall from \cite[\S6.2]{AHV} that 
for a finite group $G$, a subring $R$ of $\fie$, and
$V \in \yd{\fie G}$, an \emph{$R$-order}
$V_R$ of $V$ is a projective $R$-submodule
of $V$ satisfying $\fie \otimes_R V_R \simeq V$, $G \cdot V_R = V_R$
and $V_R = \bigoplus_{g \in G} \left(V_R \cap V_g\right)$. In what follows we will assume additionally
that $R$ is a Dedekind domain. 

\begin{lem}
\label{lem:inequality}
Let $R$ be a Dedekind domain,
$\fie$ be a field of characteristic zero,
$V\in\yd{\fie G}$ for some group $G$, 
$V_R$ an $R$-order of $V$, 
and 
$\fchar$ be a field. 
Assume that $\fchar$ is an $R$-algebra. 
For all $n\geq0$ and $g\in G$, 
\[ 
\dim \NA ^n(V)_g
\ge \dim \NAn n(\fchar\otimes_RV_R)_g 
\]
\end{lem}

\begin{proof}
Recall from \cite[\S3]{HMV} that
\[ \NA (V)=\NA (\fie \ot_R V_R)\cong \fie \ot_R \NA (V_R) \]
as $\N_0$-graded objects in $\yd{\fie G}$.
Moreover, $\fchar \ot_R \NA (V_R)$ is a connected $\N_0$-graded braided Hopf algebra
in $\yd{\fchar G}$ generated by its degree one part $\fchar \ot_R V_R$. Hence there is a surjective morphism
\[ \fchar \ot_R \NA (V_R) \twoheadrightarrow \NA (\fchar \ot_R V_R) \]
in $\yd{\fchar G}$.
Now, since $R$ is a Dedekind domain, for all $n\in \N_0$ and $g\in G$ the rank of
$\NAn n (V_R)_g $ is constant. Therefore
 \[ \dim \NAn n(\fchar \ot_R V_R)_g\le \dim \fchar \ot_R \NAn n (V_R)_g
 =\dim \fie \ot_R \NAn n (V_R)_g=\dim \NAn n (V)_g\]
 for all $n\in \N_0$ and $g\in G$. This proves the claim.
\end{proof}

Now we are ready to prove the main theorem. 

\begin{proof}[Proof of Theorem~\ref{thm:main}]
Let $X$ be the support of $V$. We may assume that $G$ is generated by $X$. 
If $V$ is not absolutely simple, then $\dim\NA(V)=\infty$ by~\cite[Corollary~3.5]{AHV}. 
Assume now that
$V$ is absolutely simple. 

Take $Y\subseteq X$ to be a subrack of $X$ isomorphic to $(12)^{\SG{4}}$, and 
let $U=\bigoplus_{y\in Y}V_y$. Let $H$ be the (finite) subgroup of $G$ generated by $Y$. 
Then $U\in\yd{\fie H}$. If $\dim V_y>1$ for some $y\in Y$, 
then $\dim\NA(U)=\infty$ by~\cite[Theorem~6.14]{AHV} and therefore $\dim\NA(V)=\infty$.

Assume next that $\dim V_y=1$ for all $y\in Y$. For the same reasons as in the previous paragraph,
$U$ has to be as in~\cite[Example~2.21]{AHV}. In particular, using the quadratic relations
(see~\cite[Example~6.4]{MR1800714} and~\cite[Definition~2.1]{MR1667680}) in $\NA (U)$, 
we may additionally assume that
\begin{equation} \label{eq:xy_noncommuting}
\dim(V_xV_y+V_{xyx^{-1}}V_x+V_yV_{xyx^{-1}})=2
\end{equation}
in $\NA(V)$ for any two non-commuting elements $x,y\in Y$, and
\begin{equation} \label{eq:xy_commute}
V_xV_y=V_yV_x\ne\{0\} 
\end{equation}
in $\NA(V)$ for any two distinct commuting elements $x,y\in Y$.

Recall from \cite[Remark~6.10]{AHV} that there exists a Dedekind domain $R\subseteq \fie $ with $R^\times \cap \Z=\{1,-1\}$ and a Yetter--Drinfeld order
$V_R\in \yd{RG}$ such that $V\simeq \fie \otimes_R V_R$.
Let $\mathfrak{m}$ be a maximal ideal of $R$ containing 3, and let $\fchar=R/\mathfrak{m}$. 
Since $R$ is a Dedekind domain and the homogeneous components of $\NA (V_R)$ are finitely generated torsion-free modules, they are projective modules \cite[Proposition~III.4.3]{Neu} and hence have constant rank \cite[Corollary~3.8]{Kun}. In particular,
\[ \dim \NA ^n(V)_g=\dim \fie \otimes_R\NA ^n(V_R)_g=\dim \fchar \otimes _R\NA ^n(V_R)_g
 \ge \dim \NA ^n(\fchar\otimes_RV_R)_g \]
 for all $n\ge 1$ and all $g\in G$ by Lemma~\ref{lem:inequality}. 
Routine calculations in degrees two and three show that $\NA (\fchar \ot_R V_R)$ 
satisfies the assumptions of
Proposition~\ref{pro:char3}, and hence it is infinite-dimensional.
Consequently, $\NA (V)$ is infinite-dimensional. This proves the theorem.
\end{proof}

\begin{rem}
    There is an alternative proof of Theorem~\ref{thm:main}
    based on an involutive automorphism and specialization to a characteristic two field. 
\end{rem}

\subsection*{Acknowledgements}
We thank Jan Ouborny for technical support in designing and implementing the AI prompt.

Part of the work  of N. A. was carried out while he held the position of Visiting Professor at the Vrije Universiteit Brussel (2025-2026). He thanks Leandro Vendramin 
for his friendly hospitality. 

N. A. also thanks Gus Lehrer, Geordie Williamson, and Ruibin Zhang, respectivley Yang Zhang, 
for conversations about $\FK{6}$ (as well as the invitation and warm welcome) 
during a visit to the Sydney Mathematical Research Institute (SMRI) in April and May 2026,
respectively to the University of Queensland in May 2026.

\bibliographystyle{abbrv}
\bibliography{refs.bib}

@article {MR2786171,
    AUTHOR = {Andruskiewitsch, N. and Fantino, F. and Gra\~na, M. and
              Vendramin, L.},
     TITLE = {Finite-dimensional pointed {H}opf algebras with alternating
              groups are trivial},
   JOURNAL = {Ann. Mat. Pura Appl. (4)},
  FJOURNAL = {Annali di Matematica Pura ed Applicata. Series IV},
    VOLUME = {190},
      YEAR = {2011},
    NUMBER = {2},
     PAGES = {225--245},
      ISSN = {0373-3114,1618-1891},
   MRCLASS = {16T05},
  MRNUMBER = {2786171},
MRREVIEWER = {Daniel\ Ion\ Bulacu},
       DOI = {10.1007/s10231-010-0147-0},
       URL = {https://doi.org/10.1007/s10231-010-0147-0},
}

@misc{AHV,
 author = {N. Andruskiewitsch and I. Heckenberger and L. Vendramin},
 title = {Pointed {Hopf} algebras of odd dimension and {Nichols} algebras over solvable groups},
 year = {2024},
 howpublished = {Preprint, {arXiv}:2411.02304 [math.{QA}] (2024)},
 url = {https://arxiv.org/abs/2411.02304},
 arXiv = {arXiv:2411.02304}
}

@article {MR5134613,
    AUTHOR = {Coulembier, Kevin and Flake, Johannes},
     TITLE = {Towards higher {F}robenius functors for symmetric tensor
              categories},
   JOURNAL = {Selecta Math. (N.S.)},
  FJOURNAL = {Selecta Mathematica. New Series},
    VOLUME = {32},
      YEAR = {2026},
    NUMBER = {5},
     PAGES = {Paper No. 98},
      ISSN = {1022-1824,1420-9020},
   MRCLASS = {20B05 (20C20)},
  MRNUMBER = {5134613},
       DOI = {10.1007/s00029-026-01204-3},
       URL = {https://doi.org/10.1007/s00029-026-01204-3},
}

@misc{2609.12895,
 author = {Thomas Creutzig and Simon D. Lentner},
 title = {Affine {$\mathfrak{sl}_2$} at admissible levels and quantum {$\mathfrak{sl}_{2|1}$}},
 year = {2026},
 howpublished = {Preprint, {arXiv}:2609.12895 [math.{QA}] (2026)},
 url = {https://arxiv.org/abs/2609.12895},
 arXiv = {arXiv:2609.12895}
}

@article {HV,
    AUTHOR = {Heckenberger, I. and Vendramin, L.},
     TITLE = {A classification of {N}ichols algebras of semisimple
              {Y}etter-{D}rinfeld modules over non-abelian groups},
   JOURNAL = {J. Eur. Math. Soc. (JEMS)},
  FJOURNAL = {Journal of the European Mathematical Society (JEMS)},
    VOLUME = {19},
      YEAR = {2017},
    NUMBER = {2},
     PAGES = {299--356},
      ISSN = {1435-9855,1435-9863},
   MRCLASS = {16T05 (17B22 20F55)},
  MRNUMBER = {3605018},
MRREVIEWER = {Cristian\ Vay},
       DOI = {10.4171/JEMS/667},
       URL = {https://doi.org/10.4171/JEMS/667},
}

@article {HMV,
    AUTHOR = {Heckenberger, I. and Meir, E. and Vendramin, L.},
     TITLE = {Finite-dimensional {N}ichols algebras of simple
              {Y}etter-{D}rinfeld modules (over groups) of prime dimension},
   JOURNAL = {Adv. Math.},
  FJOURNAL = {Advances in Mathematics},
    VOLUME = {444},
      YEAR = {2024},
     PAGES = {Paper No. 109637, 30},
      ISSN = {0001-8708,1090-2082},
   MRCLASS = {16T05 (18M15)},
  MRNUMBER = {4729697},
MRREVIEWER = {Dong\ Su},
       DOI = {10.1016/j.aim.2024.109637},
       URL = {https://doi.org/10.1016/j.aim.2024.109637},
}

@article {MR3552907,
    AUTHOR = {Blasiak, Jonah and Liu, Ricky Ini and M\'{e}sz\'{a}ros, Karola},
     TITLE = {Subalgebras of the {F}omin-{K}irillov algebra},
   JOURNAL = {J. Algebraic Combin.},
  FJOURNAL = {Journal of Algebraic Combinatorics. An International Journal},
    VOLUME = {44},
      YEAR = {2016},
    NUMBER = {3},
     PAGES = {785--829},
      ISSN = {0925-9899},
   MRCLASS = {05E15 (13D40 16S37 16T05 20F55)},
  MRNUMBER = {3552907},
MRREVIEWER = {Matt S. Davis},
       DOI = {10.1007/s10801-016-0688-4},
       URL = {https://doi.org/10.1007/s10801-016-0688-4},
}

@incollection {MR2799090,
    AUTHOR = {Andruskiewitsch, N. and Fantino, F. and Garc\'{\i}a, G. A. and
              Vendramin, L.},
     TITLE = {On {N}ichols algebras associated to simple racks},
 BOOKTITLE = {Groups, algebras and applications},
    SERIES = {Contemp. Math.},
    VOLUME = {537},
     PAGES = {31--56},
 PUBLISHER = {Amer. Math. Soc., Providence, RI},
      YEAR = {2011},
      ISBN = {978-0-8218-5239-2},
   MRCLASS = {16T05},
  MRNUMBER = {2799090},
MRREVIEWER = {Iv\'{a}n\ Ezequiel\ Angiono},
       DOI = {10.1090/conm/537/10565},
       URL = {https://doi.org/10.1090/conm/537/10565},
}

@article {MR1456522,
    AUTHOR = {Bespalov, Yu. N.},
     TITLE = {Crossed modules and quantum groups in braided categories},
   JOURNAL = {Appl. Categ. Structures},
  FJOURNAL = {Applied Categorical Structures. A Journal Devoted to
              Applications of Categorical Methods in Algebra, Analysis,
              Computer Science, Logic, Order and Topology},
    VOLUME = {5},
      YEAR = {1997},
    NUMBER = {2},
     PAGES = {155--204},
      ISSN = {0927-2852,1572-9095},
   MRCLASS = {18D10 (16W30 81R50)},
  MRNUMBER = {1456522},
MRREVIEWER = {Pilar\ C.\ Carrasco},
       DOI = {10.1023/A:1008674524341},
       URL = {https://doi.org/10.1023/A:1008674524341},
}

@article {MR1659895,
    AUTHOR = {Andruskiewitsch, N. and Schneider, H.-J.},
     TITLE = {Lifting of quantum linear spaces and pointed {H}opf algebras
              of order {$p^3$}},
   JOURNAL = {J. Algebra},
  FJOURNAL = {Journal of Algebra},
    VOLUME = {209},
      YEAR = {1998},
    NUMBER = {2},
     PAGES = {658--691},
      ISSN = {0021-8693},
   MRCLASS = {16W30},
  MRNUMBER = {1659895},
MRREVIEWER = {Stefaan Caenepeel},
       DOI = {10.1006/jabr.1998.7643},
       URL = {https://doi.org/10.1006/jabr.1998.7643},
}

@article {MR4237968,
    AUTHOR = {Etingof, Pavel and Ostrik, Victor},
     TITLE = {On the {F}robenius functor for symmetric tensor categories in
              positive characteristic},
   JOURNAL = {J. Reine Angew. Math.},
  FJOURNAL = {Journal f\"ur die Reine und Angewandte Mathematik. [Crelle's
              Journal]},
    VOLUME = {773},
      YEAR = {2021},
     PAGES = {165--198},
      ISSN = {0075-4102,1435-5345},
   MRCLASS = {18M15 (16T05)},
  MRNUMBER = {4237968},
MRREVIEWER = {Costel\ Gabriel\ Bontea},
       DOI = {10.1515/crelle-2020-0033},
       URL = {https://doi.org/10.1515/crelle-2020-0033},
}

@incollection {MR1800714,
    AUTHOR = {Milinski, Alexander and Schneider, Hans-J\"{u}rgen},
     TITLE = {Pointed indecomposable {H}opf algebras over {C}oxeter groups},
 BOOKTITLE = {New trends in {H}opf algebra theory ({L}a {F}alda, 1999)},
    SERIES = {Contemp. Math.},
    VOLUME = {267},
     PAGES = {215--236},
 PUBLISHER = {Amer. Math. Soc., Providence, RI},
      YEAR = {2000},
   MRCLASS = {16W30 (20F55)},
  MRNUMBER = {1800714},
MRREVIEWER = {Sonia Natale},
       DOI = {10.1090/conm/267/04272},
       URL = {https://doi.org/10.1090/conm/267/04272},
}

@article {MR1994219,
    AUTHOR = {Andruskiewitsch, Nicol\'{a}s and Gra\~{n}a, Mat\'{\i}as},
     TITLE = {From racks to pointed {H}opf algebras},
   JOURNAL = {Adv. Math.},
  FJOURNAL = {Advances in Mathematics},
    VOLUME = {178},
      YEAR = {2003},
    NUMBER = {2},
     PAGES = {177--243},
      ISSN = {0001-8708},
   MRCLASS = {16W30 (17B37 57M27)},
  MRNUMBER = {1994219},
MRREVIEWER = {Ian M. Musson},
       DOI = {10.1016/S0001-8708(02)00071-3},
       URL = {https://doi.org/10.1016/S0001-8708(02)00071-3},
}

@book {MR4164719,
    AUTHOR = {Heckenberger, Istv\'{a}n and Schneider, Hans-J\"{u}rgen},
     TITLE = {Hopf algebras and root systems},
    SERIES = {Mathematical Surveys and Monographs},
    VOLUME = {247},
 PUBLISHER = {American Mathematical Society, Providence, RI},
      YEAR = {[2020] \copyright 2020},
     PAGES = {xix+582},
      ISBN = {978-1-4704-5232-2},
   MRCLASS = {16T05 (16Txx 17B22)},
  MRNUMBER = {4164719},
}

@misc{Barligea,
 author = {Christoph B{\"a}rligea},
 title = {On the dimension of the {Fomin}-{Kirillov} algebra and related algebras},
 year = {2020},
 howpublished = {Preprint, {arXiv}:2001.04597 [math.{QA}] (2020)},
 url = {https://arxiv.org/abs/2001.04597},
 arXiv = {arXiv:2001.04597}
}

@article {MR5091303,
    AUTHOR = {Flake, Johannes and Laugwitz, Robert and Posur, Sebastian},
     TITLE = {Projection formulas and induced functors on centers of
              monoidal categories},
   JOURNAL = {High. Struct.},
  FJOURNAL = {Higher Structures},
    VOLUME = {10},
      YEAR = {2026},
    NUMBER = {1},
     PAGES = {126--211},
      ISSN = {2209-0606},
   MRCLASS = {18M15 (16T05 18C20)},
  MRNUMBER = {5091303},
}

@incollection {MR4486913,
    AUTHOR = {Etingof, Pavel and Ostrik, Victor},
     TITLE = {On semisimplification of tensor categories},
 BOOKTITLE = {Representation theory and algebraic geometry---a conference
              celebrating the birthdays of {S}asha {B}eilinson and {V}ictor
              {G}inzburg},
    SERIES = {Trends Math.},
     PAGES = {3--35},
 PUBLISHER = {Birkh\"{a}user/Springer, Cham},
      YEAR = {[2022] \copyright 2022},
      ISBN = {978-3-030-82006-0; 978-3-030-82007-7},
   MRCLASS = {18M20 (18M25 20D20 20J15)},
  MRNUMBER = {4486913},
MRREVIEWER = {Johannes\ Flake},
       DOI = {10.1007/978-3-030-82007-7\{_}1}

@article{zbMATH07213678,
 author = {Ostrik, Victor},
 title = {On symmetric fusion categories in positive characteristic},
 fjournal = {Selecta Mathematica. New Series},
 journal = {Sel. Math., New Ser.},
 issn = {1022-1824},
 volume = {26},
 number = {3},
 pages = {19},
 note = {Id/No 36},
 year = {2020},
 language = {English},
 doi = {10.1007/s00029-020-00567-5},
 zbMATH = {7213678},
 Zbl = {1440.18032}
}

@article {MR1774762,
    AUTHOR = {Fomin, Sergey and Procesi, Claudio},
     TITLE = {Fibered quadratic {H}opf algebras related to {S}chubert
              calculus},
   JOURNAL = {J. Algebra},
  FJOURNAL = {Journal of Algebra},
    VOLUME = {230},
      YEAR = {2000},
    NUMBER = {1},
     PAGES = {174--183},
      ISSN = {0021-8693,1090-266X},
   MRCLASS = {16W30 (14M15)},
  MRNUMBER = {1774762},
MRREVIEWER = {Harry\ Tamvakis},
       DOI = {10.1006/jabr.1999.7957},
       URL = {https://doi.org/10.1006/jabr.1999.7957},
}

@article {MR2944712,
    AUTHOR = {Vendramin, L.},
     TITLE = {Nichols algebras associated to the transpositions of the
              symmetric group are twist-equivalent},
   JOURNAL = {Proc. Amer. Math. Soc.},
  FJOURNAL = {Proceedings of the American Mathematical Society},
    VOLUME = {140},
      YEAR = {2012},
    NUMBER = {11},
     PAGES = {3715--3723},
      ISSN = {0002-9939,1088-6826},
   MRCLASS = {16T05 (16T30 20C25)},
  MRNUMBER = {2944712},
MRREVIEWER = {Cristian\ Vay},
       DOI = {10.1090/S0002-9939-2012-11215-8},
       URL = {https://doi.org/10.1090/S0002-9939-2012-11215-8},
}

@article {MR2481050,
    AUTHOR = {Andruskiewitsch, Nicol\'{a}s and Fantino, Fernando and Zhang,
              Shouchuan},
     TITLE = {On pointed {H}opf algebras associated with the symmetric
              groups},
   JOURNAL = {Manuscripta Math.},
  FJOURNAL = {Manuscripta Mathematica},
    VOLUME = {128},
      YEAR = {2009},
    NUMBER = {3},
     PAGES = {359--371},
      ISSN = {0025-2611,1432-1785},
   MRCLASS = {16W30},
  MRNUMBER = {2481050},
MRREVIEWER = {Jos\'{e}\ N.\ Alonso \'{A}lvarez},
       DOI = {10.1007/s00229-008-0237-0},
       URL = {https://doi.org/10.1007/s00229-008-0237-0},
}

@incollection {MR2106930,
    AUTHOR = {Majid, Shahn},
     TITLE = {Noncommutative differentials and {Y}ang-{M}ills on permutation
              groups {$S_n$}},
 BOOKTITLE = {Hopf algebras in noncommutative geometry and physics},
    SERIES = {Lecture Notes in Pure and Appl. Math.},
    VOLUME = {239},
     PAGES = {189--213},
 PUBLISHER = {Dekker, New York},
      YEAR = {2005},
      ISBN = {0-8247-5759-9},
   MRCLASS = {58B32 (16W35)},
  MRNUMBER = {2106930},
MRREVIEWER = {Volodymyr\ V.\ Lyubashenko},
}

@incollection {MR1667680,
    AUTHOR = {Fomin, Sergey and Kirillov, Anatol N.},
     TITLE = {Quadratic algebras, {D}unkl elements, and {S}chubert calculus},
 BOOKTITLE = {Advances in geometry},
    SERIES = {Progr. Math.},
    VOLUME = {172},
     PAGES = {147--182},
 PUBLISHER = {Birkh\"{a}user Boston, Boston, MA},
      YEAR = {1999},
   MRCLASS = {05E15 (14M15)},
  MRNUMBER = {1667680},
MRREVIEWER = {Tadeusz J\'{o}zefiak},
}

@article {MR2811166,
    AUTHOR = {Garc\'{\i}a, Gast\'{o}n Andr\'{e}s and Garc\'{\i}a Iglesias, Agust\'{\i}n},
     TITLE = {Finite-dimensional pointed {H}opf algebras over {$\mathbb S_4$}},
   JOURNAL = {Israel J. Math.},
  FJOURNAL = {Israel Journal of Mathematics},
    VOLUME = {183},
      YEAR = {2011},
     PAGES = {417--444},
      ISSN = {0021-2172},
   MRCLASS = {16T05},
  MRNUMBER = {2811166},
MRREVIEWER = {Sonia Natale},
       DOI = {10.1007/s11856-011-0055-z},
       URL = {https://doi.org/10.1007/s11856-011-0055-z},
}

@article {MR2317945,
    AUTHOR = {Andruskiewitsch, Nicol\'{a}s and Zhang, Shouchuan},
     TITLE = {On pointed {H}opf algebras associated to some conjugacy
              classes in {$\mathbb{S}_n$}},
   JOURNAL = {Proc. Amer. Math. Soc.},
  FJOURNAL = {Proceedings of the American Mathematical Society},
    VOLUME = {135},
      YEAR = {2007},
    NUMBER = {9},
     PAGES = {2723--2731},
      ISSN = {0002-9939},
   MRCLASS = {16W30},
  MRNUMBER = {2317945},
MRREVIEWER = {Gigel Militaru},
       DOI = {10.1090/S0002-9939-07-08880-6},
       URL = {https://doi.org/10.1090/S0002-9939-07-08880-6},
}

@book {Kun,
    AUTHOR = {Kunz, Ernst},
     TITLE = {Introduction to commutative algebra and algebraic geometry},
    SERIES = {Modern Birkh\"auser Classics},
      NOTE = {Translated from the 1980 German original [MR0562105] by
              Michael Ackerman,
              With a preface by David Mumford,
              Reprint of the 1985 edition [MR0789602]},
 PUBLISHER = {Birkh\"auser/Springer, New York},
      YEAR = {2013},
     PAGES = {xiv+238},
      ISBN = {978-1-4614-5986-6; 978-1-4614-5987-3},
   MRCLASS = {14-01 (13-01)},
  MRNUMBER = {2977456},
       DOI = {10.1007/978-1-4614-5987-3},
       URL = {https://doi.org/10.1007/978-1-4614-5987-3},
}

@book {Neu,
    AUTHOR = {Neukirch, J\"urgen},
     TITLE = {Algebraic number theory},
    SERIES = {Grundlehren der mathematischen Wissenschaften [Fundamental
              Principles of Mathematical Sciences]},
    VOLUME = {322},
      NOTE = {Translated from the 1992 German original and with a note by
              Norbert Schappacher,
              With a foreword by G. Harder},
 PUBLISHER = {Springer-Verlag, Berlin},
      YEAR = {1999},
     PAGES = {xviii+571},
      ISBN = {3-540-65399-6},
   MRCLASS = {11Rxx (11-02 11S15 11S31 14C40)},
  MRNUMBER = {1697859},
MRREVIEWER = {Cornelius\ Greither},
       DOI = {10.1007/978-3-662-03983-0},
       URL = {https://doi.org/10.1007/978-3-662-03983-0},
}

\end{document}